\documentclass[a4paper,oneside,10.5pt]{article}%
\usepackage{amsmath}
\usepackage{amsfonts}
\usepackage{amssymb}
\usepackage{graphicx}%
\providecommand{\U}[1]{\protect \rule{.1in}{.1in}}

\newtheorem{theorem}{Theorem}[section]

\newtheorem{definition}[theorem]{Definition}

\newtheorem{remark}[theorem]{Remark}

\newenvironment{proof}[1][Proof]{\noindent \textbf{#1.} }{\  \rule{0.5em}{0.5em}}
\numberwithin{equation}{section}

\begin{document}

\title{A New Distribution-Random Limit Normal Distribution}
\author{Xiaolin Gong \thanks{Institute for Economics and Institute of Financial
Studies, Shandong University, Jinan, Shandong 250100, PR
China(agcaelyn@gmail.com).}
\and Shuzhen Yang \thanks{School of mathematics, Shandong University, Jinan,
Shandong 250100, PR China(yangsz@mail.sdu.edu.cn).}}
\maketitle
\date{}

\textbf{Abstract: }This paper introduces a new distribution to improve tail
risk modelling. Based on the classical normal distribution, we define a new
distribution by a series of heat equations. Then, we use market data to verify
our model.

\textbf{Keywords:} fat-tail, peak, normal distribution, heat equation.

\section{Introduction}

According to Espen Gaarder Haug(2007), Wesley C. Mitchell conducted the first
empirical study of fat-tailed (high-peaked) distributions in price data as
early as 1915. Following Mandelbrot's famous 1962/63 paper on fat-tails, there
have been an abundant research focusing on the non-explained empirical fact.
Time-varying volatility, standard returns, mixture model, jump-diffusion
model, stochastic volatility, implied distributions are all important tools
for risk measurement and management. However, the 2007-08 financial crisis
exposed the deficiencies in existing risk models and reinforced the importance
of methodology improvements. More detial see \cite{Huang}, \cite{Inui},
\cite{Knight}.

In this paper, we consider a series of heat equations which relate with normal
distribution. Assuming the volatility of random varible dependent on the value
of random varible, we define a new distribution-random limit normal distribution.

The remainder of the paper followed as: Section 2 introduces a new
distribution-random limit normal distribution, and gives two usefull parameter
models. Section 3 use the random limit normal distribution to fit history
data. Some technique proof is given in the Section 4.

\section{A New Distribution}

\subsection{Random Limit Normal Distribution}

Alternatively, we propose a definition of\ normal distribution:

\begin{definition}
\label{def2.1}A -dimensional random vector $X=(X_{1},\cdots,X_{d})$\ on a
sublinear expectation space $(\Omega,\mathcal{F},P)$ is called (centralized)
normal distribution, if

$%
\begin{array}
[c]{c}%
aX+b\bar{X}\overset{d}{=}\sqrt{a^{2}+b^{2}}X,\text{ \  \ }\forall a,b>0,
\end{array}
$
\end{definition}

where $\bar{X}$\ is an independent copy of $X$, which means $\bar{X}$ is
independent of $X$, and $\bar{X}\overset{d}{=}X$. When $d=1$, we have
$X\overset{d}{=}N(0,\sigma^{2})$, where $\sigma^{2}=E\mathbb{[}X^{2}%
\mathbb{]}$.

Let $u(t,x)=E\mathbb{[}\varphi(x+\sqrt{t}X)\mathbb{]}$, then $u(t,x)$
satisfies the following heat equation:%
\begin{equation}
\partial_{t}u(t,x)-\frac{1}{2}\sigma^{2}\partial_{xx}^{2}u(t,x)=0,\text{
\  \ }u(0,x)=\varphi(x),\text{ \ }x\in R. \label{2.1}%
\end{equation}

Set $\varphi(x)=1_{\{x\leq y\}}$, we have
\begin{equation}
E\mathbb{[}\varphi(X)\mathbb{]=}P(X\leq y)=u_{y}(1,0)=\frac{1}{\sqrt
{2\pi \sigma^{2}}}\int_{-\infty}^{y}\exp(-\frac{x^{2}}{2\sigma^{2}})dx.
\label{2.2}%
\end{equation}

So the normal\ distribution of $X$ is%

\[
F_{X}(y)=u_{y}(1,0),\text{ \  \ }y\in R.
\]

Following the above definition of normal distribution, we give the definition
of\ the random limit normal function.

\begin{definition}
\label{def2.2}For a given $X$ on a probability space $(\Omega,\mathcal{F},P)$,
with function $h(\cdot):R\longrightarrow(0,+\infty),$ and the heat equation
\[
\partial_{t}u(t,x)-\frac{1}{2}\sigma^{2}h^{2}(y)\partial_{xx}^{2}%
u(t,x)=0,\text{ \  \ }u(0,y)=1_{\{x\leq y\}},\text{ \ }x\in R.
\]

We define the random limit normal function $F_{X}(\cdot)$ of $X$\ as
\begin{equation}
F_{X}(y)=u_{y}(1,0)=\frac{1}{\sqrt{2\pi(\sigma h(y))^{2}}}\int_{-\infty}%
^{y}\exp(-\frac{x^{2}}{2(\sigma h(y))^{2}})dx, \label{2.3}%
\end{equation}

where $EX=0,$ $EX^{2}=\sigma^{2}.$
\end{definition}

\begin{remark}
In general, the random limit normal function is not a distribution. However,
in the following, we will show two cases when it become a\ distribution .
\end{remark}

\begin{theorem}
\label{THe2.3}For a given $X$ on a probability space $(\Omega,\mathcal{F},P)$,
$EX=0,$ $EX^{2}=\sigma^{2}.$ For\ function $h(\cdot):R\longrightarrow
(0,+\infty),$ If any of the following case is right, $F_{X}(\cdot)$ is a distribution.

case 1: $h(\cdot)$ is continuous increasing on $(-\infty,0)$, and continuous
decreasing on $[0,+\infty);$

case 2: Set $h(y)=Ky+c,$ $c>0,$ if $y\leq0,$ $K<0,$ if $y>0,$ $K>0.$
\end{theorem}

\begin{proof}
case 1: If $y<0,$ for a given small $\delta>0,$ set $y+\delta<0.$ Then, by the
definition of $F_{X}(\cdot),$ we have%

\[%
\begin{array}
[c]{rl}%
F_{X}(y)= & \frac{1}{\sqrt{2\pi(\sigma h(y))^{2}}}\int_{-\infty}^{y}%
\exp(-\frac{x^{2}}{2(\sigma h(y))^{2}})dx,\\
F_{X}(y+\delta)= & \frac{1}{\sqrt{2\pi(\sigma h(y+\delta))^{2}}}\int_{-\infty
}^{y+\delta}\exp(-\frac{x^{2}}{2(\sigma h(y+\delta))^{2}})dx.
\end{array}
\]

Note that, $\sigma h(y+\delta)>\sigma h(y).$ By the comparing of two normal
distribution, we have%
\[
F_{X}(y+\delta)\geq F_{X}(y).
\]

Similar, if $y\geq0,$ we aslo have%
\[
F_{X}(y+\delta)\geq F_{X}(y).
\]

So $F_{X}(\cdot)$ is a increasing function on $R,$ and when $y\longrightarrow
+\infty,$ $F_{X}(y)\longrightarrow1.$

The proof of case 2 is more technique, we will show it in the Appendix.
\end{proof}

\begin{remark}
Note that, we don't need the condition $EX=0$ in the proof of Theorem
\ref{THe2.3}. From now, we use the function $h$ in the case 2, i.e.
$h(y)=K\cdot y+c$. For reader convenience, we denote $N(u,\sigma^{2},K,c)$ as
the random limit normal distribution of $X.$
\end{remark}

\section{Two Example}

\subsection{Test of Hypothesis}

Following the classical work of K. Pearson, we conduct goodness of fit test
for random limit normal distribution.

When use the method to group the data, the number of groups is $m$, while
$n$\ is the sample size. We suppose there are mutually disjoint intervals
$I_{1},\cdots,I_{m}$.%
\begin{equation}
\lambda=%
{\textstyle \sum \limits_{i=1}^{m}}
\frac{n}{p_{i}}(q_{i}-p_{i})^{2} \label{3.9}%
\end{equation}

where $q_{i}$\ is the frequency of the $i$\ th group, $1\leq i\leq m$, and
$p_{i}$\ is the corresponding probability of the fitting distribution.

K. Pearson proves the following theorem:

\begin{theorem}
If the theoretical distribution is right, when the size of sample
$n\longrightarrow \infty$, the limit distribution of statistic $\lambda$\ is
the $\chi^{2}$\ distribution in $(\Omega,\mathcal{F},P)$\ with $k-1$\ degrees
of freedom.
\end{theorem}

If the parameters in theoretical distribution are unknown, according to the
theorem of Fisher, we take the limit distribution of statistic $\lambda$ as
the $\chi^{2}$ distribution in $(\Omega,\mathcal{F},P)$\ with $k-1-r$\ degrees
of freedom. And $r$ is the number of unknown parameters.

\subsection{Afitting Example 1}

For distributions $N(\mu,\sigma^{2})$ and $N(\mu,\sigma^{2},K,c)$, we estimate
$\mu$ and $\sigma$\ by classical method. The fitting result of normal
distribution and random limit normal distribution is given. We use the test
index $\delta$\ and $\delta^{ccn}$\ to compare normal distribution and random
limit normal distribution, $\delta$\ and $\delta^{ccn}$ follow as:%

\begin{equation}
\delta=%
{\textstyle \sum \limits_{i=1}^{m}}
\frac{n}{p_{i}}(q_{i}-p_{i})^{2} \label{3.10}%
\end{equation}

\begin{equation}
\delta^{ccn}=%
{\textstyle \sum \limits_{i=1}^{m}}
\frac{n}{p_{i}}(q_{i}-p_{i}^{ccn})^{2} \label{3.11}%
\end{equation}

Where $q_{i}$\ is the frequency of the $i$ th group, $1\leq i\leq m$, $m$\ is
the number of group,\ $n$\ is the number of data, $p_{i}$ is the corresponding
probability of the fitting normal distribution, $p_{i}^{ccn}$\ is the
corresponding value of the fitting random limit normal distribution.

Using $N(\mu,\sigma^{2})$\ to fit the data of S@P 500 2000-2012 Log Return,
and the parameters is%

\[%
\begin{array}
[c]{cc}%
\text{The style of data} & \text{S@P 500}\\
\text{Length of data} & 3311\\
\mu & 1.70610051464811e-005\\
\sigma & 0.0134453931516791\\
\delta & 1129.40451908248
\end{array}
\]
and the fitting figure is

\begin{center}
\includegraphics[width=3.5 in]{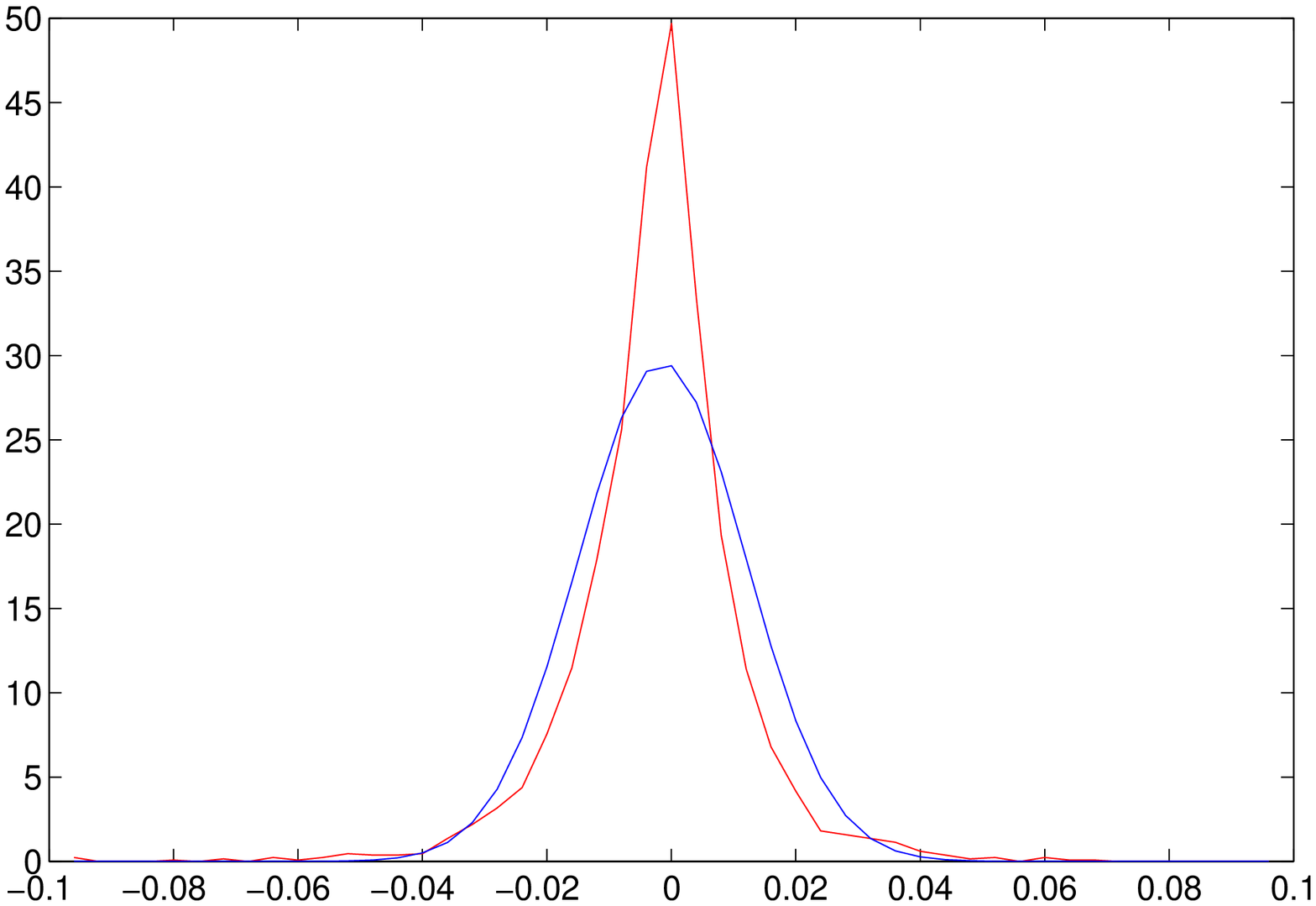}
\end{center}

Using $N(\mu,\sigma^{2},K,c)$\ to fit the data of S@P 500 2000-2012 Log
Return, and the parameters is%
\[%
\begin{array}
[c]{cc}%
\text{The style of data} & \text{S@P 500}\\
\text{Length of data} & 3311\\
\mu & 1.70610051464811e-005\\
\sigma & 0.0134453931516791\\
K & -24(y<0),24(y\geq0)\\
c & 0.5\\
\delta^{ccn} & 349.545332843509
\end{array}
\]
and the fitting figure is

\begin{center}
\includegraphics[width=3.5 in]{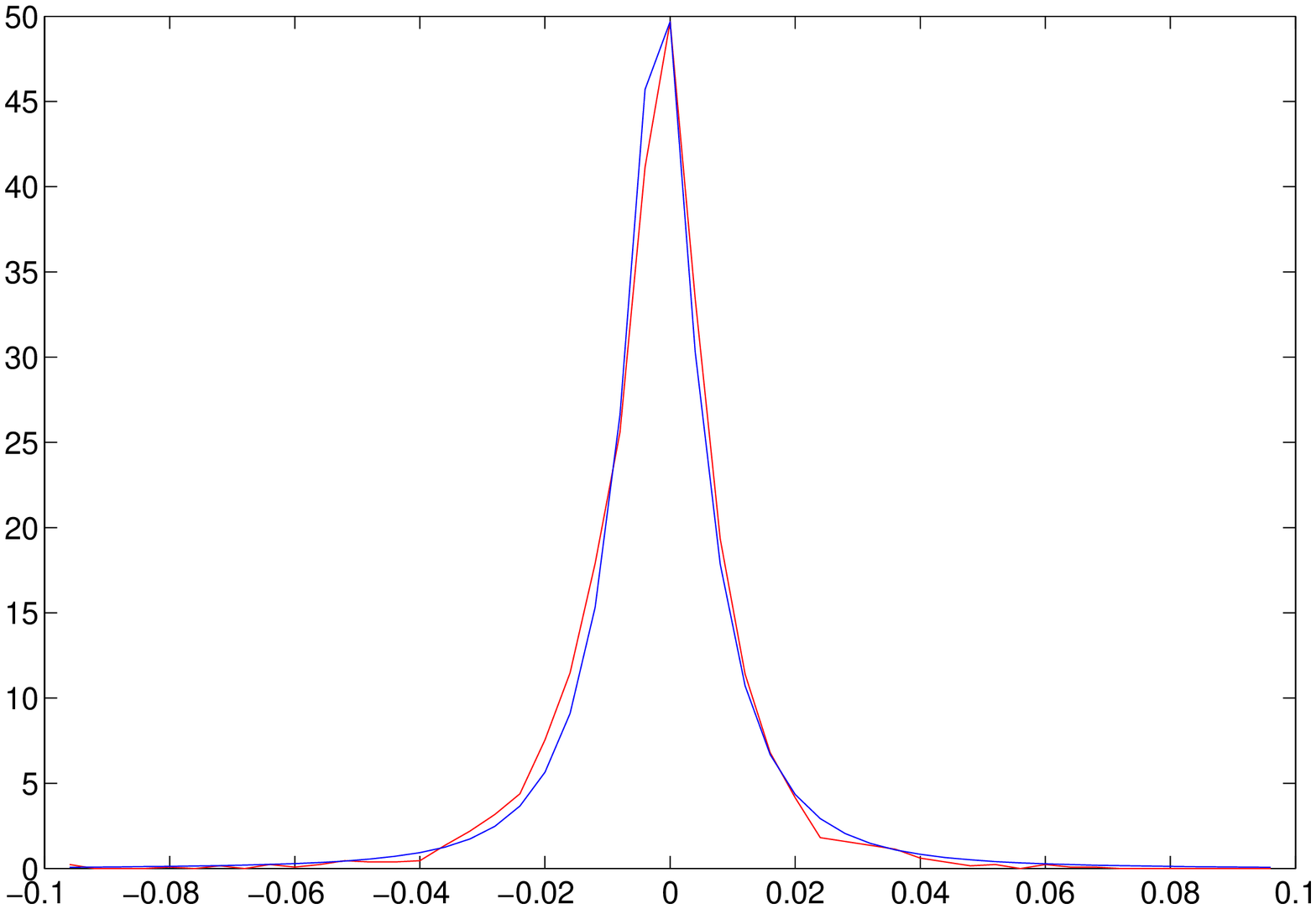}
\end{center}

Using $N(\mu,\sigma^{2})$\ to fit the data of S@P 500 1950-2012 Log Return,
the parameters is%
\[%
\begin{array}
[c]{cc}%
\text{The style of data} & \text{S@P 500}\\
\text{Length of data} & 15852\\
\mu & 1.70610051464811e-005\\
\sigma & 0.0134453931516791\\
\delta & 4830.10362738027
\end{array}
\]
and the fitting figure is

\begin{center}
\includegraphics[width=3.5 in]{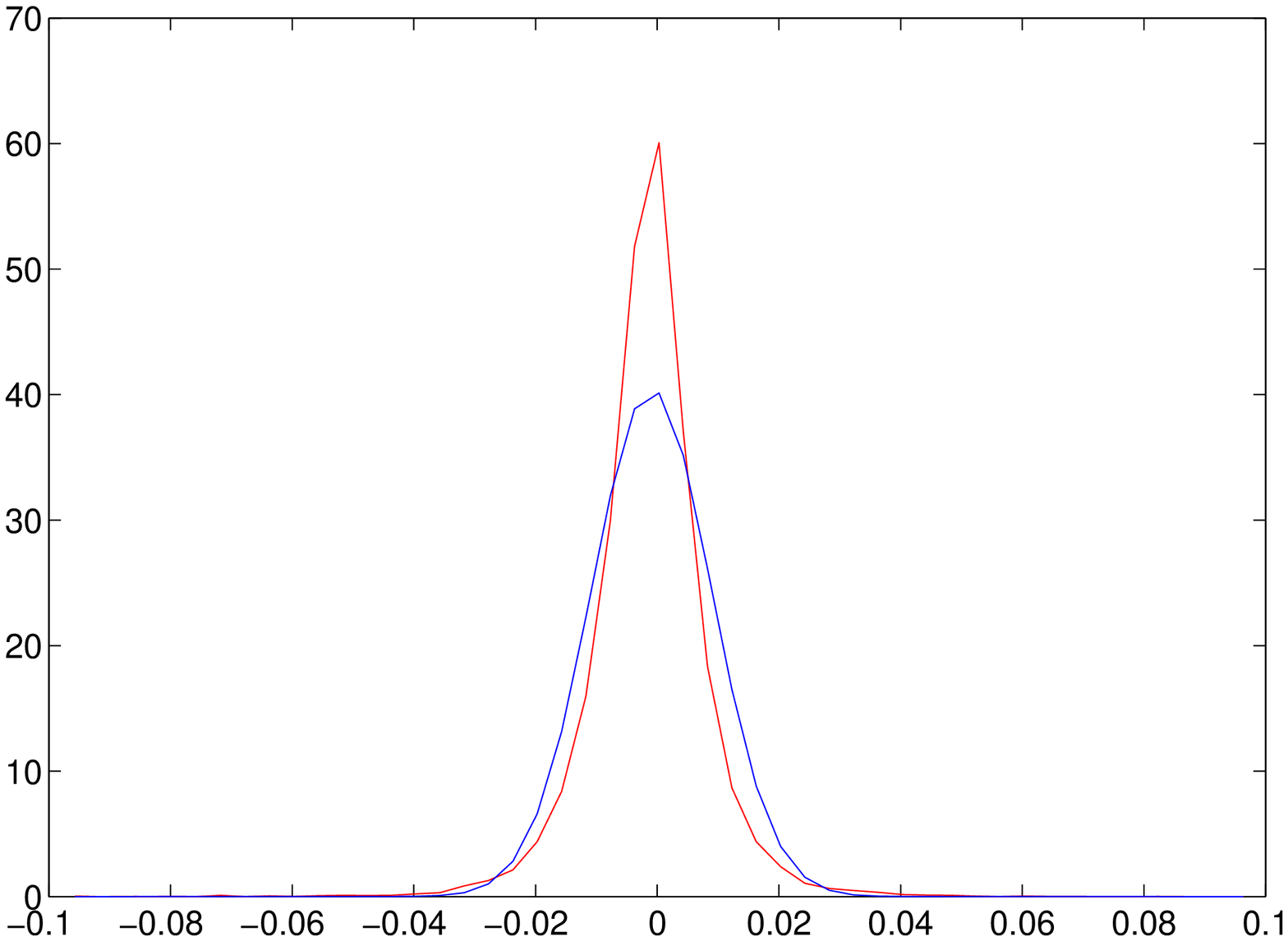}
\end{center}

Using $N(\mu,\sigma^{2},K,c)$\ to fit the data of S@P 500 1950-2012 Log
Return, the parameters is%
\[%
\begin{array}
[c]{cc}%
\text{The style of data} & \text{S@P 500}\\
\text{Length of data} & 15852\\
\mu & 1.70610051464811e-005\\
\sigma & 0.0134453931516791\\
K & -30(y<0),30(y\geq0)\\
c & 0.56\\
\delta^{ccn} & 1201.24543468315
\end{array}
\]
and the fitting figure is

\begin{center}
\includegraphics[width=3.5 in]{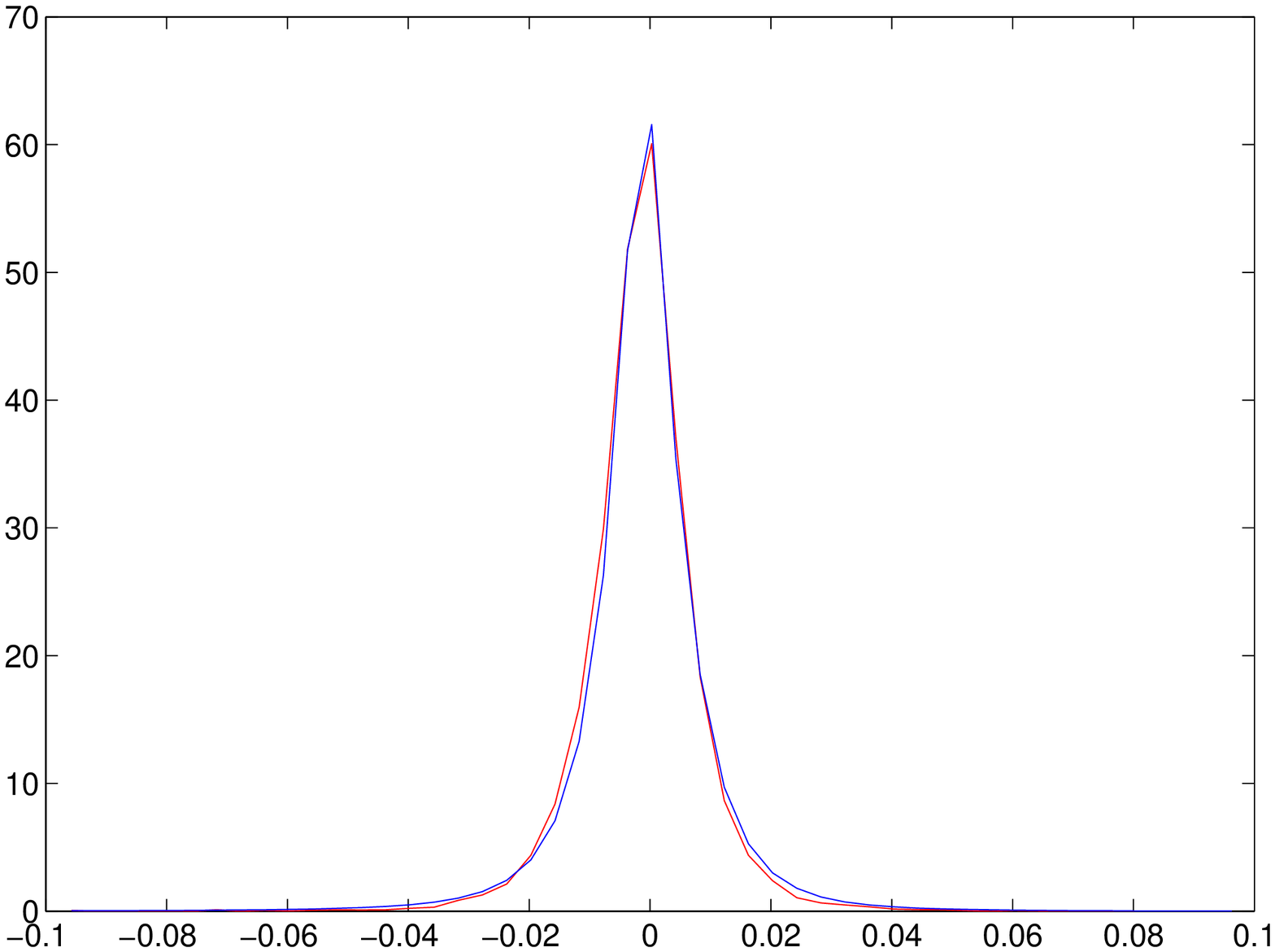}
\end{center}

\subsection{Afitting Example 2}

Using $N(\mu,\sigma^{2})$\ to fit the data of MSFT 1986-2012 Log Return, the
parameters is%
\[%
\begin{array}
[c]{cc}%
\text{The style of data} & \text{MSFT}\\
\text{Length of data} & 15852\\
\mu & 1.70610051464811e-005\\
\sigma & 0.0134453931516791\\
\delta & 2388.80170177018
\end{array}
\]
and the fitting figure is

\begin{center}
\includegraphics[width=3.5 in]{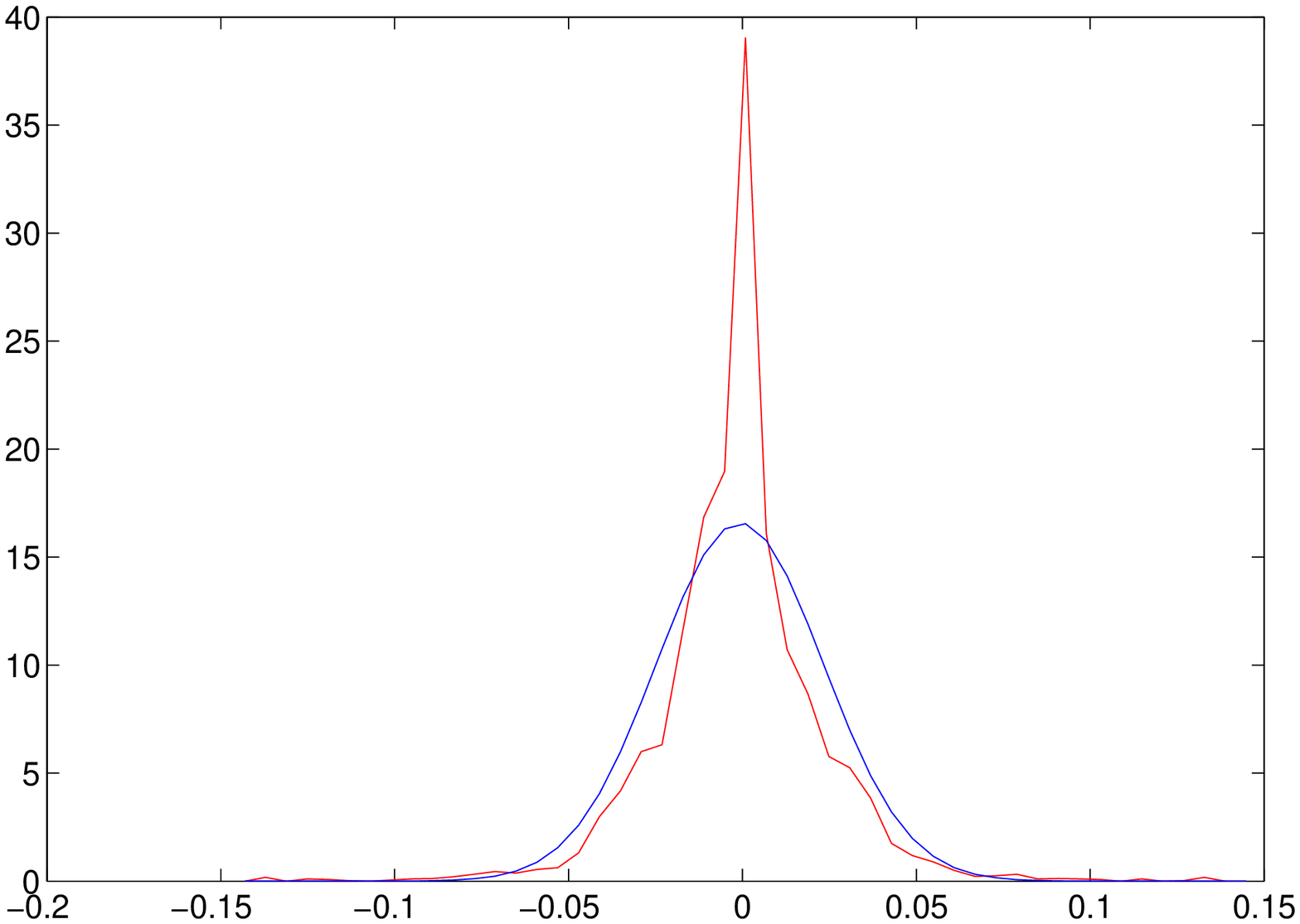}
\end{center}

Using $N(\mu,\sigma^{2},K,c)$\ to fit the data of MSFT 1986-2012 Log Return,
the parameters is%

\[%
\begin{array}
[c]{cc}%
\text{The style of data} & \text{MSFT}\\
\text{Length of data} & 15852\\
\mu & 1.70610051464811e-005\\
\sigma & 0.0134453931516791\\
K & -28.3(y<0),28.3(y\geq0)\\
c & 0.24\\
\delta^{ccn} & 1201.24543468315
\end{array}
\]
and the fitting figure is

\begin{center}
\includegraphics[width=3.5 in]{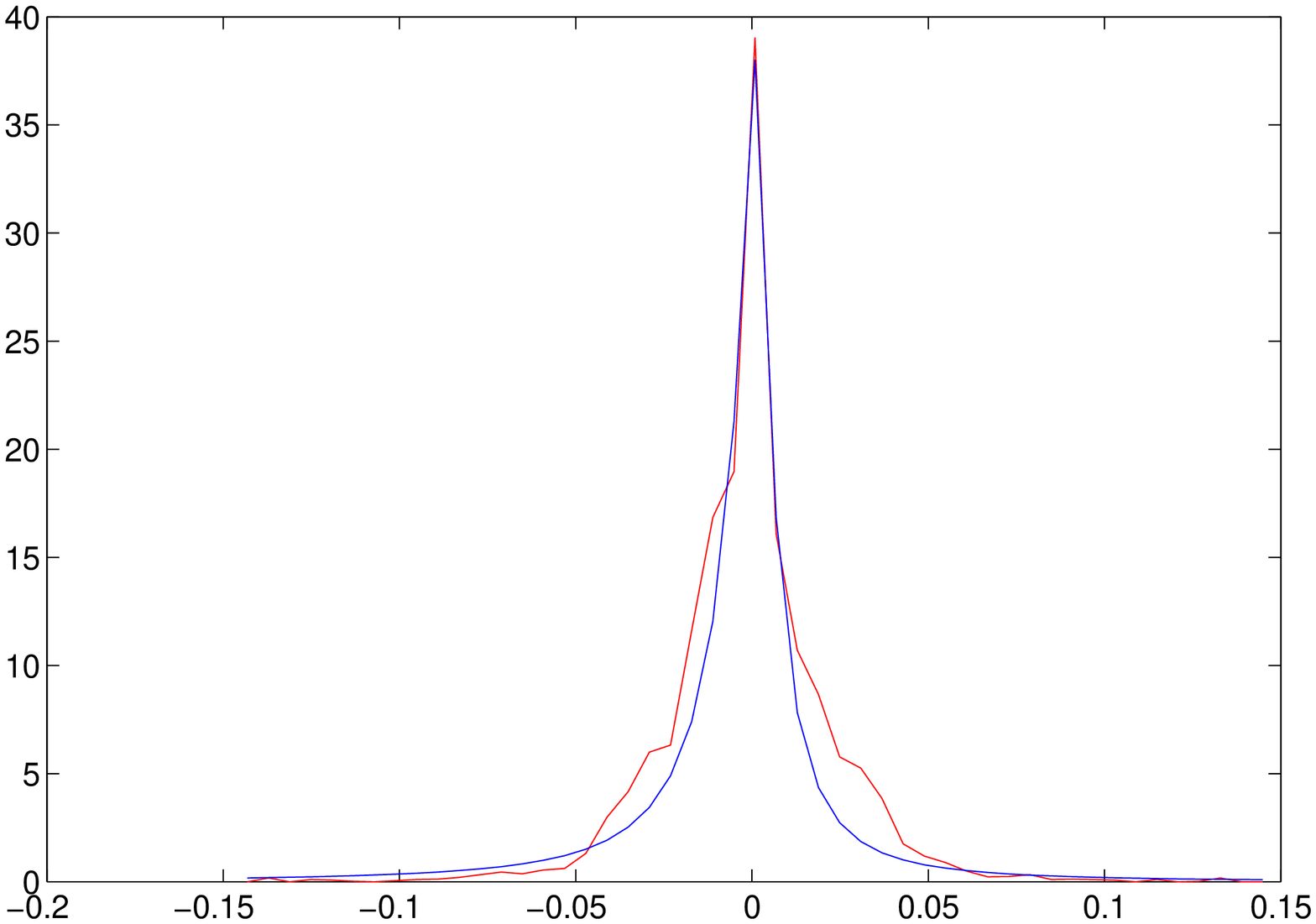}
\end{center}

Recent research is strongly motivated by the financial meltdown to develop
models and methods to better predict extreme events. In this paper, we propose
a new distribution: continuous change normal distribution which includes the
normal distribution as a special case and takes into account of main
statistically relevant stylized facts of asset returns data: fat-tail, high-peak.

\section{Appendix}

\textbf{The proof of case 2 in Theorem \ref{THe2.3}:}

\begin{proof}
case 2: For convenience, we set $\sigma=1.$

when $y\in(-\infty,0],$ $K<0.$ In order to prove $F_{X}(\cdot)$ is a
distribution$,$ we need to show that $F_{X}(\cdot)$ is a continuous and
increasing to $1^{\prime}$s function. By the definition of $F_{X}(\cdot)$, we
have%
\[%
\begin{array}
[c]{rl}%
F_{X}(y)= & \frac{1}{\sqrt{2\pi(Ky+c)^{2}}}\int_{-\infty}^{y}\exp(-\frac
{x^{2}}{2(Ky+c)^{2}})dx.
\end{array}
\]

Considering the right derivative function of $F_{X}(\cdot),$ and denote as
$f_{X}(\cdot).$%
\begin{equation}%
\begin{array}
[c]{rl}%
f_{X}(y)= & \frac{K}{\sqrt{2\pi(Ky+c)^{4}}}\int_{-\infty}^{y}[\exp
(-\frac{x^{2}}{2(Ky+c)^{2}})\cdot(\frac{x^{2}}{(Ky+c)^{2}}-1)]dx\\
& +\frac{1}{\sqrt{2\pi(Ky+c)^{2}}}\exp(-\frac{y^{2}}{2(Ky+c)^{2}}).
\end{array}
\label{4.1}%
\end{equation}

Then, we have%
\[%
\begin{array}
[c]{rl}%
f_{X}(0)= & \frac{1}{\sqrt{2\pi(Ky+c)^{2}}}-\frac{K}{2(Ky+c)}+\frac
{K}{2(Ky+c)}\\
= & \frac{1}{\sqrt{2\pi(Ky+c)^{2}}}\\
> & 0.
\end{array}
\]

Set
\[%
\begin{array}
[c]{rl}%
g_{X}(y)= & \int_{-\infty}^{y}\frac{K}{Ky+c}[\exp(-\frac{x^{2}}{2(Ky+c)^{2}%
})\cdot(\frac{x^{2}}{(Ky+c)^{2}}-1)]dx\\
& +\exp(-\frac{y^{2}}{2(Ky+c)^{2}}).
\end{array}
\]

Set $m(y)=\frac{y}{Ky+c},$ then $m\in(\frac{1}{K},0].$ By parameter
transformation, we have%

\[%
\begin{array}
[c]{rl}%
\tilde{g}_{X}(m)= & \exp(-\frac{m^{2}}{2})+\int_{-\infty}^{m}K[\exp
(-\frac{z^{2}}{2})\cdot(z^{2}-1)]dx.
\end{array}
\]

Note that, $m$ is a increasing to $0^{\prime}$s function of $y$. Also we have
$\tilde{g}_{X}(0)>0.$ The derivative function of $\tilde{g}_{X}(\cdot)$ is%
\[%
\begin{array}
[c]{rl}%
\tilde{g}_{X}^{\prime}(m)= & -\exp(-\frac{m^{2}}{2})\cdot m+K[\exp
(-\frac{m^{2}}{2})\cdot(m^{2}-1)]\\
= & \exp(-\frac{m^{2}}{2})\cdot(Km^{2}-m-K).
\end{array}
\]

For $m\in(\frac{1}{K},0],$ so $\tilde{g}_{X}^{\prime}(\cdot)>0$. Then
$\tilde{g}_{X}(\cdot)$ is increasing function of $m,$ and also $g_{X}(\cdot)$
and $f_{X}(\cdot)$ is a increasing function of $y.$

By the equation (\ref{4.1}), $f_{X}(y)\longrightarrow0,$ when
$y\longrightarrow-\infty.$ So $f_{X}(y)>0$, $y\in(-\infty,0]$, and
$F_{X}(\cdot)$ is a increasing function on $(-\infty,0].$

Similar, we could prove $F_{X}(\cdot)$ is a increasing function on $R,$ and
when $y\longrightarrow+\infty,$ $F_{X}(y)\longrightarrow1.$

This complete the proof.
\end{proof}


\section{References}

\begin{thebibliography}{9}                                                                                                %
\bibitem {Huang}Espen Gaarder, Haug., (2007), Derivatives Models on Models,
The Wiley Finance Series.

\bibitem {Inui}Inui, K. and Kijima, M., (2005), \textquotedblleft On the
significance of expected shortfall as a coherent risk
measure\textquotedblright, Journal of Banking and Finance, 29, 853-864.

\bibitem {Knight}Knight, Frank., (2006), Risk, Uncertainty, and Profit,
Mineola: Dover Publications(republication of the 1957 edition of the work
originally published in 1921 by the Houghton Mifflin Company).

\bibitem {Mises}Mises, Ludwig von., (1966), Human Action: A Treatise on
Economics, Chicago: Regnery.

\bibitem {Mandelbrot}Mandelbrot, B., (1963),The variation of certain
speculative prices, Journal of Business 26, 394-419. Mandelbrot, B., (1963),
The variation of certain speculative prices, Journalof Business 26, 394-419.
\end{thebibliography}
\end{document}